\documentclass[12pt]{amsart}

\title{Parabolic conjugacy class in finite reductive groups and its additive analogue}

\author{GyeongHyeon Nam}
\address{ Korea Institute for Advanced Study, 85 Hoegi-ro, Dongdaemun-gu, Seoul 02455, Republic of Korea}
\email{\href{mailto:ghnam@kias.re.kr}{ghnam@kias.re.kr}}

\usepackage{fullpage} 
 \usepackage[alphabetic]{amsrefs}
\usepackage{amssymb}
\usepackage{amsmath}
\usepackage{mathtools} 
\usepackage{xcolor}
\usepackage{colonequals}
\usepackage{amsrefs}
\usepackage[linktoc=all]{hyperref}
\usepackage[overload]{textcase}
\usepackage{textgreek}
\usepackage{tikz-cd}
\usepackage{mleftright}

\makeatletter
\newcommand{\vo}{\vec{o}\@ifnextchar{^}{\,}{}}
\makeatother

\theoremstyle{plain}
\newtheorem{thm}{Theorem}
\newtheorem{lem}[thm]{Lemma}
\newtheorem{prop}[thm]{Proposition}

\newtheorem{cor}[thm]{Corollary}
\theoremstyle{definition}

\theoremstyle{remark}
\newtheorem{rem}[thm]{Remark}

\newtheorem{con}[thm]{Convention}

\definecolor{red}{rgb}{1,0,0}
\definecolor{orange}{rgb}{1,0.5,0}
\definecolor{purple}{rgb}{.5,.2,.8}
\definecolor{blue}{rgb}{.2,.2,.8}
\definecolor{green}{rgb}{.4,.6,.4}

\def\bes{\begin{equation*}}  \def\ees{\end{equation*}} 
\def\bi{\begin{itemize}}   \def\ei{\end{itemize}}
\def\ba{\begin{eqnarray}} \def\ea{\end{eqnarray}}    
\def\bl{\begin{align}}    \def\el{\end{align}}       
\def\bls{\begin{align*}}    \def\els{\end{align*}}

\newcommand{\Fq}{\mathbb{F}_q}

\newcommand{\fg}{\mathfrak{g}}

\newcommand{\fp}{\mathfrak{p}}
\newcommand{\fl}{\mathfrak{l}}
\newcommand{\ft}{\mathfrak{t}}

\newcommand{\fn}{\mathfrak{n}}

\newcommand{\Ad}{{\mathrm{Ad}}}

\begin{document}

\maketitle

\begin{abstract}
%In this paper, we answer the question posed in the paper of Goodwin and R\"ohrle for reductive groups and their parabolic subgroups. 
In this paper, we answer the question posed by Goodwin and R\"ohrle for reductive groups and their parabolic subgroups.
In addition, we consider an additive analogue of this problem. By studying this additive analogue, we identify similar properties between the Deligne-Lusztig character of a finite reductive group and the Harish-Chandra induction over the corresponding finite Lie algebra.
\end{abstract}

\tableofcontents
 
\section{Introduction}
Let $G$ be a connected (untwisted) split reductive group with simply connected derived subgroup over $k=\overline{\Fq}$ and a fixed maximal split torus $T$. Let $F$ be a Frobenius map and assume that the characteristic of $k$ is very good for $G$. We denote its Lie algebra as $\fg$. For a $F$-stable parabolic subgroup $P$, we have the decomposition $P=U \rtimes L$ for its Levi subgroup $L$ and unipotent radical $U$.  For any elements $g\in G$ and $x\in \fg$, we have the Jordan decomposition $g=g_sg_u=g_ug_s$ and $x=x_s+x_n$, where semisimple elements $g_s\in G$ and $x_s\in \fg$, a unipotent element $g_u\in G$ and a nilpotent element $x_n\in \fg$. 

In this paper, we answer the question in \cite[Remark 4.14]{goodwin2009rational} over parabolic subgroups of $G$. Furthermore, we give an additive analogue of the result of \cite{goodwin2009rational} and its parabolic version. To consider the additive analogue, we used the Harish-Chandra induction on $\fg^F$, and this is an additive version of the Deligne-Lusztig character of $G^F$ from the work of Letellier \cite{letellier2004fourier}. In the additive analogue, we also observe that the Deligne-Lusztig character and the Harish-Chandra induction share common properties, for example, we decompose the trivial character or regular character in a similar way, cf. Equation   \eqref{eq:decompose-trivial} and \eqref{eq:step2}. We anticipate there are more interesting relations between the Deligne-Lusztig character and the Harish-Chandra induction.

The problem in \cite[Remark 4.14]{goodwin2009rational} is to count the number of $P^F$-conjugacy classes in $G^F$ for a $F$-stable parabolic subgroup $P$, which is denoted by $k(P^F,G^F)$. 
We consider this problem when the characteristic $p$ is very good for any pseudo-Levi subgroup $L$ (centraliser subgroups of semisimple elements in $G$) of $G$, and $q$ is sufficiently large. Note that we can assume that the characteristic $p$ is  very good for any pseudo-Levi subgroup $L$ of $G$ since there are only finitely many pseudo-Levi subsystems in the root datum of $G$. For further detailed discussion about pseudo-Levi subgroups, please see \cite[\S2.1]{KNWG}. In this paper, we denote the set of pseudo-Levi subgroups of \(G\) arising as centralisers of semisimple elements of \(G^F\) (up to \(G^F\)-conjugation) by $G_{pls}^F$, the set of unipotent elements of $G^F$  by $G_{uni}^F$, the set of representatives of unipotent conjugacy classes of $G^F$   by $[G_{uni}^F]$ and we use similar notation to nilpotent elements in its Lie algebra.

  \begin{thm}(Theorem \ref{thm:result-group})
  We have
 \[
k(P^F,G^F)=\sum_{\xi \in G_{pls}^F} \frac{|\tau_G^{-1}([\xi,1])||\xi^F|}{|W_L||G^F|} \sum_{u\in [\xi_{uni}^F]} \sum_{w\in W_L}\delta_\xi^w\sum_{v\in (W_G(T_w)/W_{\xi}(T_w))^F} Q_{\tilde{v}^{-1}T_w\tilde{v}}^{\xi}(u)
\]
where $\delta_{\xi}^w=\begin{cases}1 \quad &\text{if } g_s\in T_w\ (\text{up to conjugation}) \text{ for some }g\in \tau_G^{-1}([\xi,1])\\
0 &\text{otherwise}.
\end{cases}$
\newline
Furthermore, this satisfies the polynomial count on residue class (PORC) property.
 \end{thm}
The map $\tau_G$ is a partition map of elements in $G^F$ using the Jordan decomposition and centraliser. For the definition of $\tau_G$, please see \S\ref{s:types}, especially Equation \eqref{eq:tauG}.
\begin{con}
For convenience, we denote the Weyl group over a (fixed) split maximal torus $T$ of $G$ by $W_G$ instead of $W_G(T)$. Note that in this case, $F$-action on $W_G$ is trivial.
\end{con}
 
\bigskip
 In addition, we also consider an additive story of this problem by considering a sum \[
k(\fp^F,G^F)=\frac{1}{|{P^F}|}\sum_{x\in \fp^F} |C_{G^F}(x)|,
\]
where $C_{G^F}(x)=\{g\in G^F\,|\, Ad_g(x)=x\}$ and  $\mathfrak{p}=\mathrm{Lie}(P)$.
We give a reason to explain why this is an additive variant of $k(P^F, G^F)$ in \S\ref{sss:additive-analogue}.
 Then in this case, we have the following result.
   \begin{thm}(Theorem \ref{thm:result-additive})
 We have \[
k(\fp^F,G^F)=
 \sum_{\substack{\xi\in G_{pls}^F }} \frac{|\tau_\fg^{-1}([\xi,0])||\xi^F|}{|W_L||G^F|}\sum_{n\in [\mathrm{Lie}(\xi)_{nil}^F]} \sum_{w\in W_L}\delta_{\xi}^w  \sum_{v\in (W_G(T_w)/W_{\xi}(T_w))^F} Q_{\tilde{v}^{-1}T_w\tilde{v}}^{\xi}(\omega(n)),\]
where $\delta_{\xi}^w=\begin{cases}1 \quad &\text{if } x_s\in \ft_w\ (\text{up to adjoint action}) \text{ for some }x \in \tau_\fg^{-1}([\xi,0])\\
0 &\text{otherwise}.
\end{cases}$
\newline Furthermore, this satisfies the PORC property.
 \end{thm}
 The map $\tau_\fg$ is the partition map of elements in $\fg^F$ using the Jordan decomposition and centraliser. For the definition of $\tau_\fg$, please see \S\ref{s:types}, especially Equation \eqref{eq:taug}.

 \bigskip 
The last result is that we have an additive analogue of the result of Goodwin and R\"ohrle \cite[Proposition 4.4]{goodwin2009rational}.
 \begin{thm}[Theorem \ref{thm:GR-additive-analogue}]
For a $F$-stable parabolic subgroup $P=U\rtimes L$ with $\fn=\mathrm{Lie}(U)$, we have 
\[
 k(\mathfrak{n}^F,G^F)=\frac{ |L^F|_{p'}}{|W_L|}\sum_{w\in W_L} { (-1)^{\epsilon_L+\epsilon_{T_w}}  } \sum_{n\in\mathcal{R}(\mathfrak{n}^F,\fg^F)} Q_{T_w}^G(\omega(n)),
\]
where $\epsilon_L$ is a $F$-relative rank of $L$, and $\mathcal{R}(\fn^F,\fg^F)$ the set of representatives of adjoint orbits in $\fg^F$ that intersect $\fn^F$.
\end{thm}

\subsection{Structure of the paper}
In \S\ref{s:2} and \S\ref{sss:additive-analogue}, we present how to compute $k(P^F,G^F)$ using the character theory of $G^F$ and $k(\fp^F,G^F)$ and $k(\fn^F,G^F)$ using the Harish-Chandra induction on $\fg^F$.
In \S\ref{s:types}, we partition elements in $G^F$ and $\fg^F$.
In \S\ref{s:5} and \S\ref{s:6}, we conclude the above results.

  \section{Counting formula for reductive group}\label{s:2}
In this section, we introduce how to compute $k(P^F,G^F)$ using the character theory of $G^F$.
\subsection{A formula of $k(P^F,G^F)$}
Let $M$ be a finite group, and $N$ its subgroup. Then we define a function $f_N^M$ given by $$f_N^M(x)=| \{ mNm^{-1}\,|\,x\in mNm^{-1}\ \text{for some } m\in M\}|.$$ Then we have the following useful lemma to consider $k(P^F,G^F)$.

\begin{lem}\cite[Remark 4.14]{goodwin2009rational}
We have $$k(P^F,G^F)=\sum_{g\in \mathcal{R}(P^F,G^F)}f_{P^F}^{G^F}(g),$$ where $\mathcal{R}(P^F,G^F)$ is a set of representatives of conjugacy classes in $G^F$ that intersect $P^F$.
\end{lem}
Using \cite[Proof of Lemma 3.2]{goodwin2009rational}, we can obtain another form of $k(P^F,G^F)$ using Deligne-Lusztig characters. This lemma concerns unipotent elements, but the proof also works for arbitrary elements.
%except the last part, i.e., $f_{P}^G(u)=\frac{1}{|W_L|}\sum_{w\in W_L}Q_{T_w}^G(u)$.
\begin{prop}\cite[Proof of Lemma 3.2]{goodwin2009rational} For a $F$-stable parabolic subgroup $P$, we have 
$$f_{P^F}^{G^F}=\frac{1}{|W_L|}\sum_{w\in W_L}R_{T_w}^G(1_{T_w}),$$
where $1_{T_w}$ is the trivial character on $T_w^F$.
%where $W_L$ is the Weyl group of $L$ and $T_w$ is the $w$-twisted torus of a split maximal torus $T$.
\end{prop}
 
 By composing the above lemma and the proposition, we get the following result.
 \begin{cor}\label{coro:k}
 We have \[
k(P^F,G^F)=\sum_{g\in \mathcal{R}(P^F,G^F)} \frac{1}{|W_L|}\sum_{w\in W_L}R_{T_w}^G(1_{T_w})(g).
\]\end{cor}

\section{Counting formula for Lie algebra}\label{sss:additive-analogue} 
Let us consider an additive version of this story.   It is easy to see that \[
k(P^F,G^F)=\frac{1}{|P^F|}\sum_{p\in P^F}|C_{G^F}(p)|=\frac{1}{|P^F|}\sum_{g\in G^F}|C_{P^F}(g)|,
\]
cf. \cite[Proof of Lemma 4.1]{goodwin2009rational}.
Furthermore, it is equivalent to
\[
k(P^F,G^F)=\frac{1}{|P^F|}|\{(g,p)\in G^F\times P^F \,|\, gp=pg\}|.
\]
Then we consider an additive analogue of $k(P^F,G^F)$ by replacing a parabolic subgroup $P$ to its Lie algebra $\fp$, i.e., \[
\begin{split}
k(\fp^F,G^F)&:=\frac{1}{|P^F|}|\{ (g,x)\in G^F\times \fp^F\,|\, Ad_g(x)=x\}|\\
&= \frac{1}{|P^F|}\sum_{g\in G^F}|C_{\fp^F}(g)|=\frac{1}{|{P^F}|}\sum_{x\in \fp^F} |C_{G^F}(x)|,
\end{split}
\]
where $C_{\fp^F}(g):=\{x\in \fp^F\,|\, Ad_g(x)=x\}$ and $C_{G^F}(x)=\{g\in G^F\,|\, Ad_g(x)=x\}.$
\subsection{Parabolic subalgebra}

 Let us define a function $f_{\fp^F}^{\fg^F}: \fg^F\rightarrow \mathbb{N}\cup \{0\}$ given by $$f_{\fp^F}^{\fg^F}(x)=|\{ Ad_g(\fp^F)\,|\, x\in Ad_g(\fp^F)\ \text{for some }g\in G^F\}|,$$ and $\mathcal{R}(\fp^F,\fg^F)$ is the set of representatives of adjoint orbits in $\fg^F$ that intersect $\fp^F$.
\begin{lem}\label{lem:f_p^g}
We have $$k(\fp^F,G^F)= \sum_{x\in \mathcal{R}(\fp^F,\fg^F)}f_{\fp^F}^{\fg^F}(x).$$
\end{lem}
\begin{proof}
We use the method from \cite[Lemma 4.1]{goodwin2009rational}.  For a given $x\in \fg^F$ such that $Ad_{G^F}(x)\cap \fp^F\neq \emptyset$  (where $Ad_{G^F}(x)=\{ Ad_g(x)\,|\, g\in G^F\}$), let us consider the following set:
$$\{(y,Ad_g(\fp^F))\,|\, y\in Ad_{G^F}(x)\cap Ad_g(\fp^F)\ \text{for some }  g\in G^F\}$$   Then by computing its size
 in two different ways,
%{(y, gH) | y ∈ G· x, g ∈ G, y ∈ g H}
we can see that 
\[
|G^F:N_{G^F}(\fp^F)|\cdot | Ad_{G^F} (x)\cap \fp^F|=|Ad_{G^F}(x)| f_{\fp^F}^{\fg^F}(x).
\]
We have
\[
\begin{split}
k(\fp^F,G^F)&=\frac{1}{|P^F|}\sum_{x\in \fp^F} |\{ g\in G^F\,|\, Ad_g(x)=x\}|\\
&=\frac{1}{|P^F|}\sum_{x\in \mathcal{R}(\fp^F,\fg^F)}\sum_{\substack{Ad_h(x)\in Ad_{G^F}(x)\cap \fp^F}}  |\{ g\in G^F\,|\, Ad_{gh}(x)=Ad_{h}(x)\}|\\
&=\frac{1}{|P^F|}\sum_{x\in \mathcal{R}(\fp^F,\fg^F)} |\{ g\in G^F\,|\, Ad_{g}(x)=x\}|  \sum_{\substack{Ad_h(x)\in Ad_{G^F}(x)\cap \fp^F}} 1
\\
&=\frac{1}{|P^F|}\sum_{x\in \mathcal{R}(\fp^F,\fg^F)} |\{ g\in G^F\,|\, Ad_{g}(x)=x\}|  \cdot |Ad_{G^F}(x)\cap \fp^F|
\\
&=\frac{1}{|P^F|}\sum_{x\in \mathcal{R}(\fp^F,\fg^F)} \frac{|C_{G^F}(x)||N_{G^F}(\fp^F)||Ad_{G^F}(x)|}{|G^F|}f_{\fp^F}^{\fg^F}(x)
\\
&= \sum_{x\in \mathcal{R}(\fp^F,\fg^F)}  f_{\fp^F}^{\fg^F}(x)
\end{split}
\]
from $N_{G^F}(\fp^F)=P^F$ and orbit-stabiliser theorem.
\end{proof}
Similar to the group case, we will consider the function $f_{\fp^F}^{\fg^F}$ over the Deligne-Lusztig induction $\mathfrak{R}_{\mathfrak{l}\subset \fp}^{\fg}$ defined in \cite[Definition 3.2.13]{letellier2004fourier}. Note that this is coincides with Harish-Chandra induction defined in \cite[Definition 3.1.2]{letellier2004fourier} from \cite[Proposition 3.2.23]{letellier2004fourier} since we assume that $P$ is $F$-stable.
Note that this induction is independent on the choice of parabolic subalgebra from \cite[3.2.27]{letellier2004fourier}, so we denote that by $\mathfrak{R}_{\mathfrak{l}}^{\fg}$.

\begin{prop}\label{prop:f_p^g}
Let $\fp=\fl\oplus \mathfrak{u}$ for $\mathfrak{u}=\mathrm{Lie}(U)$. Then we have
\[
f_{\fp^F}^{\fg^F}=\frac{1}{|W_L|}\sum_{w\in W_L} \mathfrak{R}_{\ft_w}^\fg(1_{\ft_w}),
\]
where $\ft_w=\mathrm{Lie}(T_w)$ and $1_{\ft_w}$ is the constant function on $\ft_w^F$ such that $1_{\ft_w}(x)=1$ for any $x\in \ft_w^F$.\end{prop}
For convenient, we denote the constant function on $\fg^F$ (for any Lie algebra $\fg$) by $1_{\fg}$.
\begin{proof}Let us consider the canonical projection $\pi_\fp\,:\, \fp\rightarrow\fl$.
Then we have $$f_{\fp^F}^{\fg^F}=\frac{1}{|P^F|}\sum_{\substack{g\in G^F\text{ s.t.}\\Ad_{g^{-1}}(x)\in \fp^F}}1_\fp(Ad_{g^{-1}}(x))=\frac{1}{|P^F|}\sum_{\substack{g\in G^F\text{ s.t.}\\Ad_{g^{-1}}(x)\in \fp^F}}1_\fl(\pi_\fp(Ad_{g^{-1}}(x)))=\mathfrak{R}_{\fl}^{\fg}(1_\fl),$$ where the last equality comes from \cite[Definition 3.1.2]{letellier2004fourier}.

Let us recall the fact that the derived subgroup $[L,L]$ of a Levi subgroup $L$ of $G$ is again simply connected, cf. \cite[\S4.2]{HHS}.   This implies that $C_L(s)$ is connected for any semisimple element $s$ in $L$. Furthermore, $C_L(x)$ is also connected for any semisimple element $x$ in $\fl$ \cite[Proposition 2.6.18]{letellier2004fourier}. (Recall that we assume that the characteristic $p$ is very good for any Levi subgroup $L$ of $G$, and $q$ is sufficiently large.)

Now, let us show that 
\begin{equation}\label{eq:decompose-trivial}
1_{\fl}=\frac{1}{|W_L|}\sum_{w\in W_L}\mathfrak{R}_{\ft_w}^\fl(1_{\ft_w}).\end{equation}
Note that this is a similar form to the trivial character of $G^F$ with Deligne-Lusztig characters.
 Under the Jordan decomposition, let us take $x=x_s+x_n\in \fl^F$ and $g=g_sg_u\in L^F$ such that $C_L(x_s)=C_L(g_s)$  and $\omega_L(x_n)$ and $g_u$ are $L^F$-conjugate for a $L$-equivariant isomorphism $\omega_L : \fl_{nil}\rightarrow  L_{uni}$ as discussed in \cite[\S2.7.5]{letellier2004fourier}. (Note that we can find such $g$ when $q$ is sufficiently large.) 
Then from \cite[Definition 3.2.13]{letellier2004fourier} and \cite[Theorem 2.2.16]{GM20}, we can check that 
\begin{equation}\label{eq:DL-additive-formula}
\mathfrak{R}_{\ft_w}^\fl(1)(x)=\sum_{v\in (W_L(T_w)/W_{C_L(g_s)}(T_w))^F}Q_{\tilde{v}^{-1}T_w\tilde{v}}^{C_L(g_s)}(\omega_L(x_n))=\sum_{v\in (W_L(T_w)/W_{C_L(g_s)}(T_w))^F} Q_{\tilde{v}^{-1}T_w\tilde{v}}^{C_L(g_s)}(g_u)
\end{equation} and $$R_{T_w}^L(1)(g)=\sum_{v\in (W_L(T_w)/W_{C_L(g_s)}(T_w))^F} Q_{\tilde{v}^{-1}T_w\tilde{v}}^{C_L(g_s)}(g_u),$$
where $(W_L(T_w)/W_{C_L(g_s)}(T_w))^F=\{vW_{C_L(g_s)} (T_w ) \,|\,F (vW_{C_L(g_s)} (T_w )) = vW_{C_L(g_s)} (T_w )\}$ for $v\in W_L(T_1)$ and $\tilde{v}=\dot{v}\kappa \in N_L(T_w)C_L(g_s)$, cf.  \cite[Proposition 9]{nam2025multiplicity}. We have $$\{h\in L^F\,|\, Ad_h(x_s)\in \ft_w^F\}=\underset{v\in (W_L(T_w)/W_{C_L(x_s)}(T_w))^F}{\sqcup}\tilde{v}C_L(x_s)^F$$ using the same proof of \cite[Lemma 8]{nam2025multiplicity} and the fact that any centraliser subgroups of semisimple elements in $G$ and $\fg$ are connected. This is possible since $T_w\subset C_L(x_s)$ from \cite[Proposition 2.6.4]{letellier2004fourier}.
This implies that
$$\frac{1}{|W_L|}\sum_{w\in W_L}\mathfrak{R}_{\ft_w}^\fl(1_{\ft_w})(x)=\frac{1}{|W_L|}\sum_{w\in W_L}{R}_{T_w}^L(1_{T_w})(g).$$ 
%for any $g=su\in L^F$ such that $C_L(s)$ and $C_L(x_s)$ are $L$-conjugate and $\omega(x_n)=u$. 
Now, with the fact that $\frac{1}{|W_L|}\sum_{w\in W_L}{R}_{T_w}^L(1)$ is the trivial character of $L^F$, we have 
$$\frac{1}{|W_L|}\sum_{w\in W_L}\mathfrak{R}_{\ft_w}^\fl(1_{\ft_w})(x)=1$$ for any $x\in \fl^F$. Therefore, we have $$\frac{1}{|W_L|}\sum_{w\in W_L}\mathfrak{R}_{\ft_w}^\fl(1_{\ft_w})=1_\fl.$$
Using observation $1_{\fl}=\frac{1}{|W_L|}\sum_{w\in W_L}\mathfrak{R}_{\ft_w}^\fl(1_{\ft_w})$, we have the following. 
\[\begin{split}
f_{\fp^F}^{\fg^F}&=\mathfrak{R}_{\fl}^{\fg}(1_\fl)=\mathfrak{R}_{\fl}^{\fg}\left(\frac{1}{|W_L|}\sum_{w\in W_L}\mathfrak{R}_{\ft_w}^\fl(1_{\ft_w})\right)\\
&=\frac{1}{|W_L|}\sum_{w\in W_L}\mathfrak{R}_{\fl}^{\fg}(\mathfrak{R}_{\ft_w}^\fl(1_{\ft_w}))\\
&=\frac{1}{|W_L|}\sum_{w\in W_L}\mathfrak{R}_{\ft_w}^\fg(1_{\ft_w}),
\end{split}
\]
where the last equality comes from \cite[Proposition 3.2.22]{letellier2004fourier}, and recall that $\mathfrak{R}_\fl^\fg(f+g)=\mathfrak{R}_\fl^\fg(f)+\mathfrak{R}_\fl^\fg(g)$ from the definition. Therefore, we are done.
\end{proof}

\begin{rem}
It would be an interesting problem to show that $\frac{1}{|W_G|}\sum_{w\in W_G }\mathfrak{R}_{\ft_w}^\fg(1_{\ft_w})=1_\fg$ with a Lie algebra-theoretical method.
\end{rem}

\begin{cor}\label{coro:additive-k}
We have $$k(\fp^F,G^F)= \sum_{x\in \mathcal{R}(\fp^F,\fg^F)}\frac{1}{|W_L|}\sum_{w\in W_L} \mathfrak{R}_{\ft_w}^\fg(1_{\ft_w})(x).$$
\end{cor}

\subsection{Nilpotent subalgebra}
From the result of \cite{goodwin2009rational}, it would also be an interesting problem to consider $k(\mathfrak{n}^F,G^F)$. To consider this problem, we need the 
following results.

\begin{lem}\label{lem:f_n^g}
We have
\[
k(\mathfrak{n}^F,G^F)=|L^F|\sum_{x\in \mathcal{R}(\mathfrak{n}^F,\fg^F)} f_{\mathfrak{n}^F}^{\fg^F}(x).
\]
\end{lem}
\begin{proof}
The proof is the same as the proof of Lemma \ref{lem:f_p^g}.
\end{proof}

\begin{prop}\label{prop:f_n^g}
Let $\fp=\fl\oplus \mathfrak{u}$. Then we have
\[
f_{\fn^F}^{\fg^F}= \frac{1}{|W_L||L^F|_p}\sum_{w\in W_L}\frac{ (-1)^{\epsilon_L+\epsilon_{T_w}}   }{|\ft_w^F|}\sum_{y\in \ft_w^F} \mathfrak{R}_{\ft_w}^{\fg}( \mathcal{F}_{\ft_w}(\delta_y)),
\]
where $\mathcal{F}_{\ft_w}$ is the Fourier transform on $\ft_w^F$, $\epsilon_G$ is the relative $F$-rank of $G$ and the characteristic function $\delta_y$ on $\ft_w^F$, i.e., $\delta_y(x)=\begin{cases} 1\quad &\text{if }x=y \\ 0 &\text{otherwise}.
 \end{cases}$
\end{prop}
\begin{proof}
Let us define a regular function $\chi_{\fg}\,:\, \fg^F\rightarrow \mathbb{C}$ given by $\chi_{\fg}(x)=\begin{cases}
|\fg^F|\quad &\text{if }x=0\\
0&\text{otherwise}
\end{cases}
$ for any Lie algebra $\fg$,
and denote the regular character of a finite group $H$ by $\chi_H$.

\bigskip
(Step $1$) Let us show that \begin{equation}\label{eq:step1}
|\fl^F|f_{\fn^F}^{\fg^F}=\mathfrak{R}_{\fl^F}^{\fg^F}(\chi_\fl).
\end{equation}
From the definition, recall that $|\fl^F|f_{\fn^F}^{\fg^F}(x)=|\fl^F||\{ \mathrm{Ad}_{g}(\fn^F)\,|\, \Ad_g(x)\in \fn^F\text{ for some }g\in G^F \}|$ and \[\begin{split}
\mathfrak{R}_{\fl^F}^{\fg^F}(\chi_\fl)(x)&=\frac{1}{|P^F|}\sum_{\{g\in G^F\,|\, \Ad_g(x)\in \fp^F\}} \chi_\fl(\pi_\fp(\Ad_g(x)))\\
&=  \frac{1}{|P^F|}\cdot |P^F|\sum_{\{g\in G^F/P^F\,|\, \Ad_g(x)\in \fp^F\}} \chi_\fl(\pi_\fp(\Ad_g(x))) \\
&=|\fl^F||\{Ad_g(\fn^F)\,|\, \Ad_{g^{-1}}(x)\in \fn^F\text{ for some }g\in G^F\}|,
\end{split}
\]
where the first equality is \cite[Definition 3.1.2]{letellier2004fourier} and the second equality comes from $N_{G^F}(\fp^F)=P^F$.
This gives $|\fl^F|f_{\fn^F}^{\fg^F}=\mathfrak{R}_{\fl^F}^{\fg^F}(\chi_\fl)$.

\bigskip
(Step $2$) We claim that 
\begin{equation}\label{eq:step2} \chi_\fl=\frac{|\fl^F|}{|W_L||L^F|}\sum_{w\in W_L} \mathfrak{R}_{\ft_w}^{\fl}(1_{\ft_w})(0) \cdot \frac{|T_w^F|}{|\ft_w^F|}\mathfrak{R}_{\ft_w}^{\fl}(\chi_{\ft_w}).
\end{equation}
 We will prove this by computing the value of $\frac{|\fl^F|}{|W_L||L^F|}\sum_{w\in W_L} \mathfrak{R}_{\ft_w}^{\fl}(1_{\ft_w})(0) \cdot \frac{|T_w^F|}{|\ft_w^F|}\mathfrak{R}_{\ft_w}^{\fl}(\chi_{\ft_w})$ at the identity element and any non-identity elements.
 
 Let us compute the value at the identity element in $\fl^F$ as follows:
 \[\begin{split}
&\frac{|\fl^F|}{|W_L||L^F|}\sum_{w\in W_L} \mathfrak{R}_{\ft_w}^{\fl}(1_{\ft_w})(0) \cdot \frac{|T_w^F|}{|\ft_w^F|}\mathfrak{R}_{\ft_w}^{\fl}(\chi_{\ft_w})(0)\\
=&\frac{|\fl^F|}{|W_L||L^F|}\sum_{w\in W_L} \mathrm{dim}({R}_{T_w}^{L}(1_{T_w}))\cdot \frac{|T_w^F|}{|\ft_w^F|}\mathfrak{R}_{\ft_w}^{\fl}(\chi_{\ft_w})(0)
\\
=&\frac{|\fl^F|}{|W_L||L^F|}\sum_{w\in W_L} \mathrm{dim}({R}_{T_w}^{L}(1_{T_w}))\cdot \frac{|T_w^F|}{|\ft_w^F|}Q_{T_W}^L(1)|\ft_w^F|
\\
=&\frac{|\fl^F|}{ |L^F|} \frac{1}{|W_L|}\sum_{w\in W_L} \mathrm{dim}({R}_{T_w}^{L}(1_{T_w}))\cdot Q_{T_W}^L(1)|T_w^F|.
\end{split}
\]
Note that $\chi_L=\frac{1}{|W_L|}\sum_{w\in W_L} \mathrm{dim}(R_{T_w}^L(1_{T_w}))R_{T_w}^L(\chi_{T_w})$ from \cite[Proof of Lemma 3.3]{goodwin2009rational}, and this implies that
\[
\chi_{L^F}(1)=|L^F|= \frac{1}{|W_L|}\sum_{w\in W_L} \mathrm{dim}({R}_{T_w}^{L}(1_{T_w}))\cdot Q_{T_W}^L(1)|T_w^F|
\]
with the formula of the Deligne-Lusztig character, for example, \cite[Theorem 2.2.16]{GM20}.
Therefore, we get
\[
\frac{|\fl^F|}{|W_L||L^F|}\sum_{w\in W_L}  \mathrm{dim}({R}_{T_w}^{L}(1_{T_w})) \cdot \frac{|T_w^F|}{|\ft_w^F|}\mathfrak{R}_{\ft_w}^{\fl}(\chi_{\ft_w})(0)=|\fl^F|.
\]

Now, let us compute the value $\frac{|\fl^F|}{|W_L||L^F|}\sum_{w\in W_L}  \mathrm{dim}({R}_{T_w}^{L}(1_{T_w}))\cdot \frac{|T_w^F|}{|\ft_w^F|}\mathfrak{R}_{\ft_w}^{\fl}(\chi_{\ft_w})(x)$ for any $ x \in \fl^F\setminus \{0\}$. It is easy to see that $\frac{|\fl^F|}{|W_L||L^F|}\sum_{w\in W_L}  \mathrm{dim}({R}_{T_w}^{L}(1_{T_w}))\cdot \frac{|T_w^F|}{|\ft_w^F|}\mathfrak{R}_{\ft_w}^{\fl}(\chi_{\ft_w})(x)=0$ from the definition of $\chi_{\ft_w}$ if $x_s\neq 0$. If $x_s=0$, we get 
 $\frac{|\fl^F|}{|W_L||L^F|}\sum_{w\in W_L}  \mathrm{dim}({R}_{T_w}^{L}(1_{T_w})) \cdot \frac{|T_w^F|}{|\ft_w^F|}\mathfrak{R}_{\ft_w}^{\fl}(\chi_{\ft_w})(x)=\frac{|\fl^F|}{|W_L||L^F|}\sum_{w\in W_L}  \mathrm{dim}({R}_{T_w}^{L}(1_{T_w})) \cdot {|T_w^F|}Q_{T_w}^L(u)
 $ for $u=\omega_L(x_n)\neq 1 $. Then this is zero from the form of $\chi_L$ above.
 
 \bigskip
 (Step $3$) The goal of this step is to prove 
 \begin{equation}\label{eq:step3}
 \chi_{\ft_w}=\sum_{y\in \ft_w^F} \mathcal{F}_{\ft_w}(\delta_y).
 \end{equation}
  This is because
 \[\begin{split}
 \sum_{y\in \ft_w^F} \mathcal{F}_{\ft_w}(\delta_y)(x)&=\sum_{y\in \ft_w^F}\sum_{z\in \ft_w^F} \mu(\kappa(x,z))\delta_y(z)\\&=  \sum_{z\in \ft_w^F} \mu(\kappa(x,z)) 
 \end{split}
 \]
 from \cite[Page 35]{letellier2004fourier} (up to normalisation). Then since $\kappa $ is a non-degenerate bilinear form, we can conclude that
 \begin{equation}\label{eq:sum-regular-additive}
 \sum_{y\in \ft_w^F} \mathcal{F}_{\ft_w}(\delta_y)(x)=  \sum_{z\in \ft_w^F} \mu(\kappa(x,z)) =
  \begin{cases}
  |\ft_w^F|\quad &\text{if }x=0\\
  0 &\text{otherwise}
  \end{cases}
 \end{equation}
 by considering a (vector space) homomorphism $\kappa_x(z):=\kappa(x,z)$ from $\ft_w^F$ to $k^F$.
Here, $\mu$ is a non-trivial additive character of $k^F$.
 
\bigskip
(Step $4$)
From the above observations Equation \eqref{eq:step1}, \eqref{eq:step2} and \eqref{eq:step3}, we can conclude that 
\[
\begin{split}
\chi_\fl&=\frac{|\fl^F|}{|W_L||L^F|}\sum_{w\in W_L}\mathrm{dim}({R}_{T_w}^{L}(1_{T_w}))\cdot \frac{|T_w^F|}{|\ft_w^F|}\mathfrak{R}_{\ft_w}^{\fl}\left(\sum_{y\in \ft_w^F} \mathcal{F}_{\ft_w}(\delta_y)\right)\\
&=\frac{|\fl^F|}{|W_L||L^F|}\sum_{w\in W_L} \mathrm{dim}({R}_{T_w}^{L}(1_{T_w}))\cdot \frac{|T_w^F|}{|\ft_w^F|}\sum_{y\in \ft_w^F} \mathfrak{R}_{\ft_w}^{\fl}( \mathcal{F}_{\ft_w}(\delta_y))
\end{split}
\]
Then we have
\[\begin{split}
|\fl^F|f_{\fn^F}^{\fg^F}&=\mathfrak{R}_{\fl^F}^{\fg^F}(\chi_\fl)=\mathfrak{R}_{\fl^F}^{\fg^F}\left(\frac{|\fl^F|}{|W_L||L^F|}\sum_{w\in W_L} \mathrm{dim}({R}_{T_w}^{L}(1_{T_w})) \cdot \frac{|T_w^F|}{|\ft_w^F|}\sum_{y\in \ft_w^F} \mathfrak{R}_{\ft_w}^{\fl}( \mathcal{F}_{\ft_w}(\delta_y))\right)\\
&= \frac{|\fl^F|}{|W_L||L^F|}\sum_{w\in W_L} \mathrm{dim}({R}_{T_w}^{L}(1_{T_w})) \cdot \frac{|T_w^F|}{|\ft_w^F|}\sum_{y\in \ft_w^F} \mathfrak{R}_{\ft_w}^{\fg}( \mathcal{F}_{\ft_w}(\delta_y))\\
&=\frac{|\fl^F|}{|W_L||L^F|_p}\sum_{w\in W_L}\frac{ (-1)^{\epsilon_L+\epsilon_{T_w}}  }{|\ft_w^F|}\sum_{y\in \ft_w^F} \mathfrak{R}_{\ft_w}^{\fg}( \mathcal{F}_{\ft_w}(\delta_y)).
\end{split}
\]
With these steps, we are done.
\end{proof}
 
  \begin{rem}
  Note that $\sum_{y\in \ft_w^F} \mathfrak{R}_{\ft_w}^{\fg}( \mathcal{F}_{\ft_w}(\delta_y))$ has a similar property to \cite[Example 2.2.17 (b)]{GM20}.
  \end{rem}
  
  From Lemma \ref{lem:f_n^g} and Proposition \ref{prop:f_n^g}, we get the following formula.
  \begin{cor}\label{coro:nilpotent-formula}
  We have
  \[
  k(\mathfrak{n}^F,G^F)=\sum_{x\in\mathcal{R}(\mathfrak{n}^F,\fg^F)}\frac{|L^F|_{p'}}{|W_L|}\sum_{w\in W_L}\frac{ (-1)^{{\epsilon_L+\epsilon_{T_w}} }  }{|\ft_w^F|}\sum_{y\in \ft_w^F} \mathfrak{R}_{\ft_w}^{\fg}( \mathcal{F}_{\ft_w}(\delta_y))(x).
  \]
  \end{cor}
%\gncom{From the definition of $\mathcal{F}$, keep in mind that 
%\[
%\mathcal{F}_{\ft_w}(1_0)(0)=|\ft_w^F|.
%\]
%I think that this is little bit different part of \cite[The last part of the proof of Lemma 3.3]{goodwin2009rational}, i.e., 
%\[
%\sum_{\theta\in Irr(T_w^F)}R_{T_W}^G(\theta) \quad \text{vs} \quad\sum_{y\in \ft_w^F} \mathfrak{R}_{\ft_w}^{\fg}( \mathcal{F}_{\ft_w}(1_y)).
%\]For any unipotent element, every term in the left-hand side sum survive, but for any nilpotent element, $y=0$ only the survived term in the right-hand side sum. However, due to 
%$\mathcal{F}_{\ft_w}(1_0)(0)=|\ft_w^F|$, I think that $f_U^G(u)$ and $f_{\fn^F}^{\fg^F}(n)$ would be similar (up to difference between $|T_w^F|$ and $|\ft_w^F|$.)
%}

 \section{Types of elements}\label{s:types} Let us introduce types of elements in $G$. 
 \subsection{Group case}
 %For an element $g$ in $G$, we have the Jordan decomposition $g=su=us$ for a semisimple element $s$ and unipotent element $u$. 
 Let us define its type as $(C_G(g_s),g_u)$, and note that $C_G(g_s)$ is connected due to the assumption that the derived subgroup of $G$ is simply connected, and it is well-known that $g_u\in C_G(g_s)$. Then we say the set of types in $G^F$ as the type set $\Xi(G)$ of $G^F$ up to $G^F$-conjugation, i.e., \[
 \Xi(G)=\{ [C_G(s),u] \,|\, s\text{ is a semisimple element in }G^F,\ u \in C_G(s)^F\ \text{a unipotent element}\},\]
 here $[ \hspace{0.1cm} {,} \hspace{0.1cm} ]$ means the orbit under $G^F$-conjugation. Recall that this set is finite since $C_G(s)$ is determined by pseudo-Levi subsystems as noted in \cite[]{KNWG}.
 Then there is a natural map 
 \begin{equation}\label{eq:tauG}
 \tau_G\,:\, G^F\rightarrow \Xi(G).
\end{equation}
 
\subsection{Additive case}
Similarly to the group case, under the Jordan decomposition of elements in $\fg$, we  can also define types of elements in $\fg$. We denote the set of types as $\Xi(\fg)$, i.e., 
\[
\Xi(\fg):=\{ [C_G(x),n]\,|\, x \text{ is a semisimple element in }\fg^F,\ n \in C_\fg(x)^F\ \text{a nilpotent element}\}.
\]
  Therefore, we also have a natural map  \begin{equation}\label{eq:taug}\tau_\fg\,:\,\fg^F\rightarrow \Xi(\fg).
 \end{equation}
 Note that every $C_G(x)$ is always a Levi subgroup of $G$ for any semisimple element $x\in \fg$, cf. \cite[Lemma 2.6.13]{letellier2004fourier}. However, there may exist a semisimple element $s\in G$ such that $C_G(s)$ is not a Levi subgroup of $G$, cf. \cite[\S2.1]{KNWG}.
 
 \section{Group case result}\label{s:5}
 From Corollary \ref{coro:k},
we get our main result by computing the value
$$k(P^F,G^F)=\sum_{g\in \mathcal{R}(P^F,G^F)} \frac{1}{|W_L|}\sum_{w\in W_L}R_{T_w}^G(1_{T_w})(g).$$ 
 
 \subsection{Types in $\mathcal{R}(P^F,G^F)$}\label{ss:representative-group}
Consider the elements of $G^F$ that appear in $\mathcal{R}(P^F,G^F)$. 
 %Let $P=LU:=U\rtimes L$, where $L$ is the $F$-stable Levi subgroup $L$ and the unipotent radical $U$ of $P$. 
 Let us take $g\in G$ such that $G.g \cap P^F\neq 0$. 
 %Then under the Jordan decomposition, we have $g=su=us$ such that $u \in C_G(s)$. 
 %Without loss of the generality, $g\in P$, then $s\in L$ (up to $P$-conjugation). 
 %\gncom{I think that this is true, so I need to find a reference about this fact.} This implies that $s \in T_w$ for some $w\in W_L$. Therefore, we have the following result.
 \begin{lem}
The types of elements in $\mathcal{R}(P^F,G^F)$ are $[C_G(s),u]$, where $s\in T_w$ for some $w\in W_L$.
 \end{lem}
We denote the set of types that appear in $\mathcal{R}(P^F,G^F)$ by $\Xi_L(G)$.
 \begin{proof}
 It is enough to show that for a semisimple element $s\in P^F$, there exists $g\in P^F$ such that $gsg^{-1}\in L^F$. Since $P$ is an algebraic group, any two maximal tori in $P$ are $P$-conjugate.
Let us assume that $s$ is in a maximal torus $T'$ of $P$, and take another maximal torus $T$ in $L$. Then there exists $p\in P$ such that $pT'p^{-1}=T$.  This implies  $psp^{-1}\in L$. Without loss of generality, we can assume that $p$ is an element in $U$. This is because if $p=p_lp_u$ for some $p_l\in L $ and $p_u \in U$, then $p_lp_uT'(p_lp_u)^{-1}=T\Rightarrow p_uT'p_u^{-1}=p_lTp_l^{-1}\subset L$.
  Let us decompose $s$ by $s_ls_u$ for $s_l\in L$ and $s_u \in U$. Then $s_l\in L^F$ and $s_u \in U^F$ since $F(s)=F(s_l s_u)=s_ls_u=s\Rightarrow s_l^{-1}F(s_l)=s_uF(s_u)^{-1}\in L\cap U=\{1\}$.
  
Now, let us consider the projection map $\pi_L\,:\, P \rightarrow L$, and it is easy to check that this is a homomorphism since $U$ is normal in $P$. Then since $psp^{-1}\in L$, we have
\[
psp^{-1}=\pi_L(psp^{-1})=\pi_L(p)\pi_L(s)\pi_L(p^{-1})=1\cdot s_l \cdot 1=s_l.
\]
Then since $F(s_l)=s_l$, we have
\[
F(p)sF(p^{-1})=F(psp^{-1})= s_l=psp^{-1}\Rightarrow p^{-1}F(p)\in C_U(s) .
\]
  Note that  $C_U(s)$ is connected from \cite[\S18.3 Theorem]{Hum}. Then by applying Lang's theorem (cf. \cite[Theorem 1.4.8]{GM20}), there exists $u\in C_U(s)$ (such that $u\neq p$ since the Lang map is not bijective) such that
\[
u^{-1}F(u)=p^{-1}F(p)\Rightarrow pu^{-1}=F(pu^{-1})\Rightarrow pu^{-1}\in U^F.
\]
Then we have
\[
pu^{-1}sup^{-1}=psp^{-1}=s_l\in L.
\]
Therefore, we are done.
 \end{proof}
 
 \subsection{Conclusion}
 We can conclude the following:
 \begin{thm}\label{thm:result-group}
 We have
 \[
k(P^F,G^F)=\sum_{\xi \in G_{pls}^F} \frac{|\tau_G^{-1}([\xi,1])||\xi^F|}{|W_L||G^F|} \sum_{u\in [\xi_{uni}^F]} \sum_{w\in W_L}\delta_\xi^w\sum_{v\in (W_G(T_w)/W_{\xi}(T_w))^F} Q_{\tilde{v}^{-1}T_w\tilde{v}}^{\xi}(u),
\]
where $\delta_{\xi}^w=\begin{cases}1 \quad &\text{if } g_s\in T_w\ (\text{up to conjugation}) \text{ for some }g\in \tau_G^{-1}([\xi,1])\\
0 &\text{otherwise}.
\end{cases}$
\newline
Furthermore, this satisfies the PORC property.
 \end{thm}
 \begin{proof}We have
 \[
\begin{split}
k(P^F,G^F)&=\sum_{g\in \mathcal{R}(P^F,G^F)} \frac{1}{|W_L|}\sum_{w\in W_L}R_{T_w}^G(1_{T_w})(g)\\
&=\sum_{\substack{[\xi,u]\in \Xi_L(G) }}\sum_{g\in \tau_G^{-1}([\xi,u])} \frac{|C_\xi(u)^F|}{|G^F|} \frac{1}{|W_L|}\sum_{w\in W_L}R_{T_w}^G(1_{T_w})(g)
\\
&
=\sum_{\substack{[\xi,u]\in \Xi_L(G) }} \frac{|\tau_G^{-1}([\xi,u])||C_\xi(u)^F|}{|W_L||G^F|}\sum_{w\in W_L}\delta_\xi^w \sum_{v\in (W_G(T_w)/W_{\xi}(T_w))^F} Q_{\tilde{v}^{-1}T_w\tilde{v}}^{\xi}(u).
\end{split}
\] 
Note that if $g_s$ is not in $T_w$ (up to $G^F$-conjugation), then $R_{T_w}^G(1)(g)=0$ (cf. \cite[Example 2.2.17 (a)]{GM20}). Furthermore,  $g_s\in T_w $ (up to conjugation) for some $g\in \tau_G^{-1}([\xi,u])$ if and only if $y_s\in T_w$ (up to $G^F$-conjugation) for every $y\in \tau_G^{-1}([\xi,u])$. This is because for any semisimple element $s \in G$, $C_G(s)$ contains every maximal torus containing $s$. 
Using $G_{pls}^F$ and $[\xi_{uni}]$, we can get the following result:
\[
\begin{split}
&\sum_{\substack{[\xi,u]\in \Xi_L(G) }} \frac{|\tau_G^{-1}([\xi,u])||C_\xi(u)^F|}{|W_L||G^F|}\sum_{w\in W_L}\delta_\xi^w \sum_{v\in (W_G(T_w)/W_{C_G(g_s)}(T_w))^F} Q_{\tilde{v}^{-1}T_w\tilde{v}}^{\xi}(u)\\
&=\sum_{\xi \in G_{pls}^F} \frac{|\tau_G^{-1}([\xi,1])||\xi^F|}{|W_L||G^F|} \sum_{u\in \xi_{uni}^F} \frac{|C_\xi(u)^F|}{|\xi^F|}  \sum_{w\in W_L}\delta_\xi^w\sum_{v\in (W_G(T_w)/W_{\xi}(T_w))^F} Q_{\tilde{v}^{-1}T_w\tilde{v}}^{\xi}(u)\\
&=\sum_{\xi \in G_{pls}^F} \frac{|\tau_G^{-1}([\xi,1])||\xi^F|}{|W_L||G^F|} \sum_{u\in [\xi_{uni}^F]} \sum_{w\in W_L}\delta_\xi^w\sum_{v\in (W_G(T_w)/W_{\xi)}(T_w))^F} Q_{\tilde{v}^{-1}T_w\tilde{v}}^{\xi}(u).
\end{split}
\]

%The PORC property follows from the PORC property of   $|\tau^{-1}([\xi,1])|$ for each pseudo-Levi subgroup $\xi$ and the term $\sum_{v\in (W_G(T_w)/W_{C_G(g_s)}(T_w))^F} Q_{\tilde{v}^{-1}T_w\tilde{v}}^{C_G(g_s)}(g_u)$ is independent on $q$. 

%The PORC property of $|C_\xi(u)^F|$ is proved in \cite[Proposition 3.3]{GR} when the centre of $\xi$ is connected. When the centre of $\xi$ is disconnected, we can check that $|C_\xi(u)^F|=|(Z(\xi)^\circ)^F|m_\xi(\overline{u},F) |C_{\overline{\xi}}(\overline{u})^F|$ where $\overline{\xi}=\xi/Z(\xi)$,$\overline{u}=\pi(u)$ for $\pi\,:\, \xi\rightarrow \xi/Z(\xi)$ and $m_\xi(\overline{u},F)$ is the number of distinct $\xi^F$-conjugacy classes corresponding to the $\overline{\xi}^F$-conjugacy classes of $\overline{u}$. This equation comes from the fact that each $\xi^F$-conjugacy classes corresponding to the $\overline{\xi}^F$-conjugacy classes of $\overline{u}$ have the same size. Then from \cite[\S2.3]{goodwin2009rational}, the number $m_\xi(\overline{u},F)$ only depends on the residue of $q$ modulo some positive integer, and so this implies the PORC property of $|C_\xi(u)^F|$.
 Let us show that this value gives the PORC property of $k(P^F,G^F)$.
We have the PORC property of   $\tau^{-1}([\xi,1])$   (cf. \cite{Der}) and the number of unipotent conjugacy classes only depends on the congruence of modulo some positive integer using \cite[\S2.3]{goodwin2009rational} again when $q$ is a power of a good prime for every $\xi$ (recall that there are only finitely many pseudo-Levi subgroups up to $G$-conjugation). In addition, the value of the Green function and the size $|W_\xi(T_w)^F|$ is independent on $q$ (for any pseudo-Levi subgroup $\xi$ of $ G$), cf. \cite{shoji} (or \cite[Theorem 3.10]{goodwin2009rational}) and \cite[Proposition 3.3.6]{Carter}. 
Let us check the independence over $q$ of the term $\sum_{v\in (W_G(T_w)/W_{C_G(g_s)}(T_w))^F} Q_{\tilde{v}^{-1}T_w\tilde{v}}^{C_G(g_s)}(g_u)$. From \cite[Lemma 8, Proposition 9]{nam2025multiplicity}, $\tilde{v}=\dot{v}\kappa \in G^F$. Note that $F$-action on $W_L(T_w)$ is given by $F(v)=wvw^{-1}$, cf. \cite[Page 26]{DM}. Then $\tilde{v}^{-1}T_w\tilde{v}$ corresponds to $\kappa F(\kappa^{-1})=\dot{v}^{-1}F(\dot{v})=\dot{v}^{-1}\dot{w}\dot{v}\dot{w}^{-1}$, and so this is independent on $q$. As the last one, this is polynomial applying \cite[Lemma 2.11]{goodwin2009rational}.
Therefore, with considering such observations, we can check that $k(P^F,G^F)$ satisfies the PORC property.
 \end{proof}

 \section{Additive case result}\label{s:6}
We conclude this paper by presenting a result concerning the additive analogue problem.
 
\subsection{Types in $\mathcal{R}(\fp^F,\fg^F)$}
Let us consider an additive analogue of \S\ref{ss:representative-group} with the same notation.
\begin{lem}
The types of elements in $\mathcal{R}(\fp^F,\fg^F)$ are $[C_G(x),n]$, where $x\in \ft_w$ for some $w\in W_L$.
\end{lem}
Similarly, we denote the set of types that appear in $\mathcal{R}(\fp^F,\fg^F)$ by $\Xi_\fl(\fg)$.
\begin{proof}
It is enough to show that for a semisimple element $x\in\fp^F$, there exists $p\in P^F$ such that $Ad_p(x)\in \fl^F$.
Note that $Ad_u(x)-x\in\mathfrak{n}$ for any $u \in U$ and $x\in \mathfrak{l}$. This is because $lul^{-1}u^{-1}\in U$ for any $u \in U$ and $l \in L$. \footnote{More explicitly, let us take $\phi_u(l):=lul^{-1}u^{-1}$. Then from \cite[Proposition 2.1.1]{KNP}, we have 
$d\phi_u(y)=Ad_l(1-Ad_u(dl_l^{-1}y))\in \mathfrak{n}$ for any $y\in \fl$. This implies that $(1-Ad_u)(dl_l^{-1}y)\in \mathfrak{n}$, and so we get $Ad_u(x)-x\in \mathfrak{n}$ for any $u \in U$ and $x\in \fl$.}

Now,  as the group case, we can assume that 
 $y:=Ad_{p}(x)\in \fl$ for some $p\in U$, and we have
\[
F(y)=F(Ad_{p }(x))=Ad_{F(p )}(F(x))=Ad_{F(p)}(x)=Ad_{F(p)}Ad_{p^{-1}}(y)=Ad_{F(p)p^{-1}}(y)\in \fl.
\]
Then we have \[
F(y)-y=Ad_{F(p)p^{-1}}(y)-y\in \mathfrak{n}\cap \fl=\{0\}
\]
from the above discussion (since $F(p)p^{-1}\in U$). Then we have $F(y)=y$ and $Ad_{F(p)p^{-1}}(y)=y\Rightarrow F(p)p^{-1}\in C_U(y)$. From \cite[Lemma 2.6.5]{letellier2004fourier}, $C_U(y)$ is connected, and so we can apply Lang's theorem. Then for $v\in C_U(y)$ such that $v^{-1}F(v)=pF(p^{-1})$, we have that $ F(vp)=vp\in U^F$. This implies that 
\[
Ad_{vp}(x)=Ad_{v}(Ad_p(x))=Ad_v(y)=y,
\]
and so we get $x$ is in $\fl^F$ up to $P^F$-adjoint action.
\end{proof}

\subsection{Result}We can conclude the following:
  \begin{thm}\label{thm:result-additive}
 We have \[
k(\fp^F,G^F)=
 \sum_{\substack{\xi\in G_{pls}^F }} \frac{|\tau_\fg^{-1}([\xi,0])||\xi^F|}{|W_L||G^F|}\sum_{n\in [\mathrm{Lie}(\xi)_{nil}^F]} \sum_{w\in W_L}\delta_{\xi}^w  \sum_{v\in (W_G(T_w)/W_{\xi}(T_w))^F} Q_{\tilde{v}^{-1}T_w\tilde{v}}^{\xi}(\omega(n)),\]
where $\delta_{\xi}^w=\begin{cases}1 \quad &\text{if } x_s\in \ft_w\ (\text{up to adjoint action}) \text{ for some }x \in \tau_\fg^{-1}([\xi,0])\\
0 &\text{otherwise}.
\end{cases}$
\newline 
Furthermore, this satisfies the PORC property.
\end{thm}
  
\begin{proof}
Recall that we have $$k(\fp^F,G^F)= \sum_{x\in \mathcal{R}(\fp^F,\fg^F)}\frac{1}{|W_L|}\sum_{w\in W_L} \mathfrak{R}_{\ft_w}^\fg(1_{\ft_w})(x)$$ from Corollary \ref{coro:additive-k}. Then we can prove the formula with the same method in Theorem \ref{thm:result-group} using Equation \eqref{eq:DL-additive-formula}.

%Note that if $x_s$ is not in $\ft_w$ (up to adjoint action), then $\mathfrak{R}_{\ft_w}^\fg(1)(x)=0$ (cf. \cite[Definition 3.2.13]{letellier2004fourier}). Furthermore,  $x_s\in \ft_w $ (up to adjoint action) for some $x\in \tau_\fg^{-1}([\xi,n])$ if and only if $y_s\in \ft_w$ (up to the adjoint action) for every $y\in \tau_\fg^{-1}([\xi,n])$. This is because for any semisimple element $s \in \fg$, $C_\fg(s)$ contains every Lie algebra of maximal torus containing $s$, cf. \cite[Proposition 2.6.4]{letellier2004fourier}. Then with Equation \eqref{eq:DL-additive-formula}, we can conclude this theorem.

The PORC property comes from   $G$-equivariant isomorphism $\omega$, the polynomiality result about $\tau_\fg^{-1}([\xi,0])$ in \cite[Theorem 3.2]{Fle} and  \cite[Lemma 2.11]{goodwin2009rational}.   Therefore, we are done.
\end{proof}

%\begin{rem}
%The PORC property of $k(\fp^F,G^F)$ depends on the PORC property of $|\tau_\fg^{-1}([\xi,n])|$, and to the best of our knowledge, this remains unknown. This problem is also posed in \cite[Question 4.2.7]{NP}. However, when the characteristic is regular as defined in \cite{lehrer1992rational}, the size $|\tau_\fg^{-1}([\xi,1])|$ is known in \cite[Theorem 4.11]{lehrer1992rational}.
%\end{rem}

Let us check an additive analogue of the result of Goodwin and R\"ohrle \cite[Proposition 4.4]{goodwin2009rational} as follows.
\begin{thm}\label{thm:GR-additive-analogue}
We have 
\[
 k(\mathfrak{n}^F,G^F)=\frac{ |L^F|_{p'}}{|W_L|}\sum_{w\in W_L} { (-1)^{\epsilon_L+\epsilon_{T_w}}  } \sum_{n\in\mathcal{R}(\mathfrak{n}^F,\fg^F)} Q_{T_w}^G(\omega(n)).
\]
\end{thm}
This formula has the same form as the formula for the group case \cite[Proposition 4.3]{goodwin2009rational}, and so this obvisouly satisfies the PORC property without using \cite[Theorem 3.2]{Fle}.
\begin{proof}
Recall Corollary \ref{coro:nilpotent-formula}, i.e.,
  \[
  k(\mathfrak{n}^F,G^F)=\sum_{n\in\mathcal{R}(\mathfrak{n}^F,\fg^F)}\frac{ |L^F|_{p'}}{|W_L|}\sum_{w\in W_L}\frac{ (-1)^{\epsilon_L+\epsilon_{T_w}}  }{|\ft_w^F|}\sum_{y\in \ft_w^F} \mathfrak{R}_{\ft_w}^{\fg}( \mathcal{F}_{\ft_w}(\delta_y))(n).
  \]
Then for a nilpotent element $n$, we have \[
\begin{split}\sum_{y\in \ft_w^F} \mathfrak{R}_{\ft_w}^{\fg}( \mathcal{F}_{\ft_w}(\delta_y))(n)&=\sum_{y\in \ft_w^F} Q_{T_w}^G(\omega(n)) \mathcal{F}_{\ft_w}(\delta_y)(0)\\
&=Q_{T_w}^G(\omega(n))\sum_{y\in \ft_w^F}  \mathcal{F}_{\ft_w}(\delta_y)(0)\\
&=Q_{T_w}^G(\omega(n))|\ft_w^F|
\end{split}\]
from Equation \eqref{eq:sum-regular-additive}.
Therefore, we get
\[
  k(\mathfrak{n}^F,G^F)=\frac{ |L^F|_{p'}}{|W_L|}\sum_{w\in W_L} { (-1)^{\epsilon_L+\epsilon_{T_w}}  } \sum_{n\in\mathcal{R}(\mathfrak{n}^F,\fg^F)} Q_{T_w}^G(\omega(n)).
\]
This finishes the proof.
\end{proof}

\bigskip
\noindent \textbf{Acknowledgments}		
GyeongHyeon Nam was supported by Oscar Kivinen's Väisälä project grant of the Finnish Academy of Science and Letters and a KIAS Individual Grant
(HP111501) at Korea Institute for Advanced Study.

\end{document}